\documentclass[hidelinks]{siamart251216}

\usepackage{mathtools}
\usepackage{amssymb}
\usepackage{amsopn}
\usepackage{microtype}
\usepackage{enumitem}
\usepackage{booktabs}

\newsiamthm{conjecture}{Conjecture}
\newsiamremark{remark}{Remark}
\newsiamremark{example}{Example}

\DeclareMathOperator{\Hilb}{Hilb}

\newcommand{\kk}{\mathbf{k}}

\newcommand{\D}{\mathcal{D}}

\title{Betti Numbers of Cut Complexes of Squared Paths and a Recurrence Conjecture}
\author{Yutong Zhang\thanks{Corresponding author. School of Mathematics, Sichuan University, 24 First Loop Road South Section I, Chengdu, Sichuan 610064, China. Email: \email{yutongzhang@stu.scu.edu.cn}.}
\and Yaoran Yang\thanks{School of Mathematics, Sichuan University, 24 First Loop Road South Section I, Chengdu, Sichuan 610064, China. Email: \email{yangyaoran@stu.scu.edu.cn}.}}
\headers{BETTI NUMBERS OF CUT COMPLEXES OF SQUARED PATHS}{Y. ZHANG AND Y. YANG}
\hypersetup{
  pdftitle={Betti Numbers of Cut Complexes of Squared Paths and a Recurrence Conjecture},
  pdfauthor={Yutong Zhang and Yaoran Yang},
  pdfsubject={Cut complexes of squared paths, Betti numbers, and finite-difference recurrences},
  pdfkeywords={cut complex, squared path, Betti number, shellable complex, finite difference, algebraic combinatorics}
}

\begin{document}

\maketitle

\begin{abstract}
For a graph $G$ on $[n]$, the $k$-cut complex $\Delta_k(G)$ has facets $[n]\setminus T$, where $T$ ranges over the $k$-subsets that induce disconnected subgraphs of $G$. Bayer, Denker, Jeli\'c Milutinovi\'c, Sundaram, and Xue proved that $\Delta_k(P_n^2)$ is shellable for $n\ge k+3$ and conjectured a finite-difference recurrence for its top reduced Betti number along each diagonal $n-k=r$. We prove the exact formula
\[
 \beta(k,n)=\binom{n-1}{k-1}
 -\sum_{j=0}^{\min\{k-1,n-k\}}\binom{k-1}{j}(n-k-j+1)+(n-k)
\]
for $k\ge2$ and $n\ge k+3$. The proof rests on a complete classification of complements of cardinality at least $k$ that contain no disconnected $k$-subset: every such complement has size $k$ or $k+1$ and is, respectively, a connected $k$-subset of $P_n^2$ or an interval of $k+1$ consecutive vertices. For fixed $r\ge3$, the formula is the restriction of a polynomial in $k$ of degree exactly $r-1$. Hence its $r$th backward difference vanishes identically and its $(r-1)$st backward difference is the constant $r-2$. On the topologically defined sequence, these identities hold for $k\ge r+2$ and $k\ge r+1$, respectively. We also obtain the conjectured closed forms for $k=4,5$, the complete face enumerator, and the associated $h$-polynomial and Stanley--Reisner Hilbert series.
\end{abstract}

\begin{keywords}
cut complex, squared path, Betti number, shellable complex, finite difference, algebraic combinatorics
\end{keywords}

\begin{MSCcodes}
05E45, 05C69, 05A15, 13F55
\end{MSCcodes}

\section{Introduction}\label{sec:introduction}

Cut complexes of graphs lie at the interface of graph theory, algebraic combinatorics, and topological combinatorics. They are naturally related to Alexander-dual constructions for squarefree monomial ideals; see \cite{Froberg1985,EagonReiner1998} and the subsequent development of cut complexes in \cite{BayerTotal2024,BayerCut2024,BayerCutII2025}. Our purpose is to resolve the finite-difference recurrence conjectured in \cite[Conjecture 4.10]{BayerCutII2025} for the top reduced Betti numbers of cut complexes of squared paths.

Let $G$ be a finite simple graph on $[n]=\{1,\ldots,n\}$. For $2\le k\le n$, its $k$-cut complex is
\begin{equation}\label{eq:cut-complex}
 \Delta_k(G)=\big\langle [n]\setminus T:T\in\tbinom{[n]}{k},\,G[T]\text{ is disconnected}\big\rangle.
\end{equation}
If no disconnected $k$-subset exists, we regard $\Delta_k(G)$ as the void complex. Otherwise every facet has cardinality $n-k$. The squared path $P_n^2$ is the graph on $[n]$ in which $i$ and $j$ are adjacent exactly when $1\le |i-j|\le2$.

Bayer et al.\ proved that, for $k\ge2$ and $n\ge k+3$, the complex $\Delta_k(P_n^2)$ is shellable and homotopy equivalent to a wedge of spheres of dimension $n-k-1$ \cite[Theorem 4.4]{BayerCutII2025}. Its reduced homology is therefore concentrated in top degree. For a field $\kk$, set $\beta(k,n)=\dim_{\kk}\widetilde H_{n-k-1}(\Delta_k(P_n^2);\kk)$ for $k\ge2$ and $n\ge k+3$. This number is independent of $\kk$, because the cited theorem determines the homotopy type.

For fixed $r=n-k\ge3$, the conjecture of Bayer et al.\ asserts that
\begin{equation}\label{eq:original-recurrence}
 \beta(k,k+r)=\sum_{i=1}^{r}(-1)^{i-1}\binom{r}{i}\beta(k-i,k-i+r),
 \qquad k\ge r+3.
\end{equation}
Our principal result is the following exact enumeration. Throughout, for $a\in\mathbb Z_{\ge0}$ and $b\in\mathbb Z$, we use the convention $\binom ab=0$ when $b<0$ or $b>a$.

\begin{theorem}[Exact Betti formula]\label{thm:main}
Let $k\ge2$, let $n\ge k+3$, and put $r=n-k$. Then
\begin{equation}\label{eq:main-formula}
 \beta(k,n)=\binom{n-1}{k-1}
 -\sum_{j=0}^{\min\{k-1,r\}}\binom{k-1}{j}(r-j+1)+r.
\end{equation}
Equivalently,
\begin{equation}\label{eq:main-diagonal-formula}
 \beta(k,k+r)=\binom{k+r-1}{r}
 -\sum_{j=0}^{r}(r-j+1)\binom{k-1}{j}+r.
\end{equation}
\end{theorem}

The proof is organized around complements. If $F\subseteq[n]$ and $C=[n]\setminus F$, then $F$ is a face of $\Delta_k(P_n^2)$ if and only if $C$ contains a disconnected $k$-subset. We classify the complementary sets for which every $k$-subset is connected. There are only two nontrivial cardinality levels: a bad complement of size $k$ is a connected $k$-subset, a bad complement of size $k+1$ is an interval of $k+1$ consecutive vertices, and no bad complement has size at least $k+2$. This classification determines the entire face enumerator and hence the reduced Euler characteristic. The wedge-of-spheres theorem then converts that Euler characteristic into \eqref{eq:main-formula}.

For fixed $r\ge3$, the right-hand side of \eqref{eq:main-diagonal-formula} is a polynomial in $k$ of degree exactly $r-1$. The resulting finite-difference statement is slightly stronger than the original conjecture: the recurrence \eqref{eq:original-recurrence} is valid whenever all Betti numbers appearing in it are defined, namely for $k\ge r+2$. Moreover, the $(r-1)$st backward difference equals $r-2$ on its natural topological range $k\ge r+1$. Thus order $r$ is sharp among annihilating operators that are pure powers of the backward-difference operator.

The formula also proves \cite[Conjecture 4.11]{BayerCutII2025}. Namely, for $n\ge7$,
\begin{equation}\label{eq:intro-k4}
 \beta(4,n)=3+8\binom{n-7}{1}+6\binom{n-7}{2}+\binom{n-7}{3},
\end{equation}
and, for $n\ge8$,
\begin{equation}\label{eq:intro-k5}
 \beta(5,n)=6+20\binom{n-8}{1}+21\binom{n-8}{2}
 +7\binom{n-8}{3}+\binom{n-8}{4}.
\end{equation}

Section~\ref{sec:complements} develops the complement criterion and classifies all bad complements. Section~\ref{sec:enumeration} derives the face enumerator, the reduced Euler characteristic, and the exact Betti formula. Section~\ref{sec:diagonals} proves the recurrence, determines its optimal order, and gives the precise one-sided generating-function formulation. Section~\ref{sec:consequences} records the fixed-parameter formulas, the layerwise description of nonfaces, and the Stanley--Reisner enumerators. Section~\ref{sec:conclusion} concludes.

\section{Complements and connected subsets}\label{sec:complements}

\subsection{The complementary-set criterion}

For a graph $G$ on $V=[n]$, let $\D_k(G)$ denote the family of $k$-subsets $T\subseteq V$ for which $G[T]$ is disconnected. Then \eqref{eq:cut-complex} may be written as $\Delta_k(G)=\langle V\setminus T:T\in\D_k(G)\rangle$.

\begin{lemma}[Complement criterion]\label{lem:complement-criterion}
Let $G$ be a graph on $V=[n]$, let $2\le k\le n$, and let $F\subseteq V$. Then
\[
 F\in\Delta_k(G)
 \quad\Longleftrightarrow\quad
 \text{there exists }T\in\tbinom{V\setminus F}{k}\text{ such that }G[T]\text{ is disconnected}.
\]
\end{lemma}

\begin{proof}
If $F\in\Delta_k(G)$, then, by the definition of a generated simplicial complex, $F$ is contained in a generating face $V\setminus T$ for some $T\in\D_k(G)$. Hence $T\subseteq V\setminus F$, and $T$ is the required disconnected $k$-subset. Conversely, if a disconnected $k$-subset $T$ is contained in $V\setminus F$, then $F\subseteq V\setminus T$, so $F$ belongs to $\Delta_k(G)$.
\end{proof}

A set $C\subseteq V$ is called \emph{$k$-bad} if $|C|\ge k$ and every $k$-subset of $C$ induces a connected subgraph of $G$. For every integer $m$, let $q_m(G,k)$ be the number of $k$-bad subsets of $V$ having cardinality $m$. In particular, $q_m(G,k)=0$ for $m<k$ and for $m>n$.

Let $F_{\Delta}(x)=\sum_{F\in\Delta}x^{|F|}$ be the face enumerator of a simplicial complex $\Delta$. For $0\le p\le n-k$, Lemma~\ref{lem:complement-criterion} gives $f_{p-1}(\Delta_k(G))=\binom np-q_{n-p}(G,k)$, where $f_{p-1}$ counts faces of cardinality $p$. Consequently,
\begin{equation}\label{eq:general-face-enumerator}
 F_{\Delta_k(G)}(x)=\sum_{p=0}^{n-k}\left(\binom np-q_{n-p}(G,k)\right)x^p.
\end{equation}
Whenever $\Delta_k(G)$ is nonvoid, it contains the empty face; with the usual empty-face convention, $\widetilde\chi(\Delta_k(G))=-F_{\Delta_k(G)}(-1)$.

\subsection{Connected subsets of a squared path}

Write a nonempty subset of $[n]$ as $A=\{a_1<\cdots<a_t\}$. The following elementary criterion is the basis of all subsequent counting.

\begin{lemma}[Gap criterion]\label{lem:gap-criterion}
The induced graph $P_n^2[A]$ is connected if and only if $a_{i+1}-a_i\le2$ for every $1\le i<t$. Equivalently, it is disconnected if and only if $a_{i+1}-a_i\ge3$ for at least one $i$.
\end{lemma}

\begin{proof}
If every consecutive gap is at most two, then each pair $a_i,a_{i+1}$ is an edge of $P_n^2$. Hence $a_1-a_2-\cdots-a_t$ is a path contained in the induced graph, which is therefore connected.

Conversely, suppose that $a_{i+1}-a_i\ge3$ for some $i$. If $u\in\{a_1,\ldots,a_i\}$ and $v\in\{a_{i+1},\ldots,a_t\}$, then $v-u\ge a_{i+1}-a_i\ge3$. Thus no edge of $P_n^2[A]$ joins the two nonempty sets $\{a_1,\ldots,a_i\}$ and $\{a_{i+1},\ldots,a_t\}$. The induced graph is disconnected.
\end{proof}

\begin{lemma}[Connected $k$-subsets]\label{lem:connected-count}
For $2\le k\le n$, the number $z_{k,n}$ of connected $k$-subsets of $P_n^2$ is
\begin{equation}\label{eq:zkn}
 z_{k,n}=\sum_{j=0}^{\min\{k-1,n-k\}}\binom{k-1}{j}(n-k-j+1).
\end{equation}
\end{lemma}

\begin{proof}
Let $A=\{a_1<\cdots<a_k\}$ be connected. By Lemma~\ref{lem:gap-criterion}, every gap $d_i=a_{i+1}-a_i$ is either $1$ or $2$. Suppose that exactly $j$ of the $k-1$ gaps are equal to $2$. Their positions can be chosen in $\binom{k-1}{j}$ ways, and then $a_k-a_1=(k-1)+j$, so $a_k=a_1+k+j-1$. Thus $A\subseteq[n]$ if and only if $1\le a_1\le n-k-j+1$. For each gap pattern there are therefore $n-k-j+1$ possible starting values, provided $j\le n-k$. Summing over $0\le j\le\min\{k-1,n-k\}$ proves \eqref{eq:zkn}. The construction is reversible, so no connected set is omitted or counted more than once.
\end{proof}

\subsection{Classification of bad complements}

\begin{proposition}[Complete bad-complement profile]\label{prop:bad-profile}
Let $2\le k\le n-2$. For every integer $m$,
\begin{equation}\label{eq:bad-profile}
 q_m(P_n^2,k)=
 \begin{cases}
 0, & m<k,\\[1mm]
 z_{k,n}, & m=k,\\[1mm]
 n-k, & m=k+1,\\[1mm]
 0, & m\ge k+2,
 \end{cases}
\end{equation}
where $z_{k,n}$ is given by \eqref{eq:zkn}.
\end{proposition}

\begin{proof}
The case $m<k$ follows directly from the definition of a $k$-bad set. We prove the remaining cases separately.

If $|C|=k$, then $C$ has exactly one $k$-subset, namely $C$ itself. Hence $C$ is $k$-bad if and only if $P_n^2[C]$ is connected. Lemma~\ref{lem:connected-count} therefore gives $q_k(P_n^2,k)=z_{k,n}$.

Now suppose $|C|=k+1$ and write $C=\{c_1<\cdots<c_{k+1}\}$. Assume first that $C$ is $k$-bad. For each $1\le i\le k-1$, delete $c_{i+1}$. In the resulting $k$-subset, the vertices $c_i$ and $c_{i+2}$ are consecutive in increasing order. That $k$-subset must be connected, so Lemma~\ref{lem:gap-criterion} implies $c_{i+2}-c_i\le2$. Since $c_i<c_{i+1}<c_{i+2}$ are distinct integers, one also has $c_{i+2}-c_i\ge2$. Therefore $c_{i+2}=c_i+2$ and necessarily $c_{i+1}=c_i+1$. This holds for every $i$, and hence $c_{i+1}=c_i+1$ for $1\le i\le k$. Thus $C$ is an interval $\{s,s+1,\ldots,s+k\}$.

Conversely, deleting one vertex from an interval of $k+1$ consecutive integers leaves a $k$-subset whose consecutive gaps are all $1$, except possibly for one gap equal to $2$. Lemma~\ref{lem:gap-criterion} shows that every such $k$-subset is connected. Hence the interval is $k$-bad. There are $n-k$ possible intervals, so $q_{k+1}(P_n^2,k)=n-k$.

Finally, let $|C|=k+s$ with $s\ge2$, and write $C=\{c_1<\cdots<c_{k+s}\}$. Take $T=\{c_1\}\cup\{c_{s+2},c_{s+3},\ldots,c_{k+s}\}$. The second set has $k-1$ elements, so $|T|=k$. Its two smallest elements are $c_1$ and $c_{s+2}$. Since the $s$ elements $c_2,\ldots,c_{s+1}$ lie strictly between them, $c_{s+2}-c_1\ge s+1\ge3$. By Lemma~\ref{lem:gap-criterion}, $P_n^2[T]$ is disconnected. Hence $C$ is not $k$-bad, and $q_m(P_n^2,k)=0$ for every $m\ge k+2$.
\end{proof}

\section{Face enumeration and exact Betti numbers}\label{sec:enumeration}

Throughout this section, assume $2\le k\le n-2$ and put $r=n-k\ge2$. The set $\{1,\ldots,k-1,n\}$ induces a disconnected subgraph, because the gap between its last two vertices is $n-(k-1)=r+1\ge3$. Thus $\Delta_k(P_n^2)$ is nonvoid and has a facet of cardinality $r$. Every generating face has this cardinality, so the complex has dimension $r-1$. The bad-complement profile therefore determines its entire face set.

\begin{theorem}[Face enumerator]\label{thm:face-enumerator}
Let $2\le k\le n-2$ and $r=n-k$. Then
\begin{equation}\label{eq:face-enumerator}
 F_{k,n}(x):=\sum_{F\in\Delta_k(P_n^2)}x^{|F|}
 =\sum_{p=0}^{r}\binom np x^p-rx^{r-1}-z_{k,n}x^r.
\end{equation}
Equivalently,
\begin{equation}\label{eq:f-vector}
 f_{p-1}\bigl(\Delta_k(P_n^2)\bigr)=
 \begin{cases}
 \binom np, & 0\le p\le r-2,\\[1mm]
 \binom n{r-1}-r, & p=r-1,\\[1mm]
 \binom nr-z_{k,n}, & p=r.
 \end{cases}
\end{equation}
\end{theorem}

\begin{proof}
Apply \eqref{eq:general-face-enumerator} to $G=P_n^2$. A face of cardinality $p$ has a complement of cardinality $n-p$. By Proposition~\ref{prop:bad-profile}, a nonzero correction can occur only when $n-p=k+1$ or $n-p=k$, equivalently when $p=r-1$ or $p=r$. At $p=r-1$, the correction is $q_{k+1}(P_n^2,k)=n-k=r$; at $p=r$, it is $q_k(P_n^2,k)=z_{k,n}$. This gives \eqref{eq:face-enumerator}, and coefficient extraction gives \eqref{eq:f-vector}.
\end{proof}

\begin{corollary}[Lower skeleton and top nonfaces]\label{cor:lower-skeleton}
Under the hypotheses of Theorem~\ref{thm:face-enumerator}, the complex $\Delta_k(P_n^2)$ contains the full $(r-3)$-skeleton of the simplex on $[n]$, with the standard convention that the full $(-1)$-skeleton consists of the empty face. Its nonfaces of cardinality $r-1$ are exactly the complements of intervals of $k+1$ consecutive vertices, and its nonfaces of cardinality $r$ are exactly the complements of connected $k$-subsets of $P_n^2$.
\end{corollary}

\begin{proof}
The equality $f_{p-1}=\binom np$ for $p\le r-2$ proves the first assertion. If $|F|=r-1$, then $|[n]\setminus F|=k+1$, and Proposition~\ref{prop:bad-profile} identifies the bad complements as the consecutive intervals. If $|F|=r$, then its complement has cardinality $k$ and is bad exactly when it is connected.
\end{proof}

We next compute the reduced Euler characteristic. We use the elementary partial-sum identity
\begin{equation}\label{eq:partial-binomial-sum}
 \sum_{p=0}^{d}(-1)^p\binom np=(-1)^d\binom{n-1}{d},
 \qquad 0\le d<n.
\end{equation}

\begin{proposition}[Euler characteristic]\label{prop:euler}
Let $2\le k\le n-2$ and $r=n-k$. Then
\begin{equation}\label{eq:euler}
 \widetilde\chi\bigl(\Delta_k(P_n^2)\bigr)
 =(-1)^{r-1}\left(\binom{n-1}{k-1}-z_{k,n}+r\right).
\end{equation}
\end{proposition}

\begin{proof}
By the empty-face convention and Theorem~\ref{thm:face-enumerator},
\begin{align*}
 \widetilde\chi\bigl(\Delta_k(P_n^2)\bigr)
 &=-F_{k,n}(-1)\\
 &=-\sum_{p=0}^{r}(-1)^p\binom np+r(-1)^{r-1}+z_{k,n}(-1)^r.
\end{align*}
Because $k\ge2$, one has $r<n$, so \eqref{eq:partial-binomial-sum} applies with $d=r$. Hence
\begin{align*}
 \widetilde\chi\bigl(\Delta_k(P_n^2)\bigr)
 &=-(-1)^r\binom{n-1}{r}+r(-1)^{r-1}+z_{k,n}(-1)^r\\
 &=(-1)^{r-1}\left(\binom{n-1}{r}-z_{k,n}+r\right).
\end{align*}
Finally, $\binom{n-1}{r}=\binom{n-1}{k-1}$ because $r=n-k$.
\end{proof}

The only external topological input needed for the main formula is the following theorem.

\begin{theorem}[Bayer--Denker--Jeli\'c Milutinovi\'c--Sundaram--Xue \cite{BayerCutII2025}]\label{thm:external}
If $k\ge2$ and $n\ge k+3$, then $\Delta_k(P_n^2)$ is shellable and homotopy equivalent to a wedge of spheres of dimension $n-k-1$.
\end{theorem}

\begin{proof}[Proof of Theorem~\ref{thm:main}]
Let $r=n-k\ge3$. By Theorem~\ref{thm:external}, all reduced homology groups of $\Delta_k(P_n^2)$ vanish except possibly in degree $r-1$, and the top group is free. Hence $\widetilde\chi(\Delta_k(P_n^2))=(-1)^{r-1}\beta(k,n)$. Combining this identity with Proposition~\ref{prop:euler} and cancelling $(-1)^{r-1}$ gives $\beta(k,n)=\binom{n-1}{k-1}-z_{k,n}+r$. Substitution of the value of $z_{k,n}$ from Lemma~\ref{lem:connected-count} proves \eqref{eq:main-formula}. Replacing $n$ by $k+r$ gives \eqref{eq:main-diagonal-formula}; its upper summation limit may be written as $r$ because $\binom{k-1}{j}=0$ for $j>k-1$.
\end{proof}

The same Euler-characteristic expression also covers the boundary $r=2$, although $\beta(k,n)$ was defined above only for $r\ge3$. Indeed, \cite[Example 4.3]{BayerCutII2025} shows that $\Delta_k(P_{k+2}^2)$ is a contractible one-dimensional complex, and direct substitution gives $\binom{k+1}{2}-3-2(k-1)-\binom{k-1}{2}+2=0$. For $r=3$, Theorem~\ref{thm:main} yields $\beta(k,k+3)=\binom{k-1}{2}$; for $k\ge3$, this agrees with \cite[Theorem 4.9]{BayerCutII2025}. These boundary checks also verify the normalization of the reduced Euler characteristic.

Because Theorem~\ref{thm:external} gives an actual wedge of spheres, the integral top homology is free. Thus $\dim_{\kk}\widetilde H_{r-1}(\Delta_k(P_n^2);\kk)=\operatorname{rank}_{\mathbb Z}\widetilde H_{r-1}(\Delta_k(P_n^2);\mathbb Z)$ for every field $\kk$.

\section{Diagonal polynomials, finite differences, and generating functions}\label{sec:diagonals}

\subsection{The diagonal polynomial}

For an indeterminate $x$ and an integer $m\ge0$, write $\binom{x}{m}=x(x-1)\cdots(x-m+1)/m!$, with the empty product interpreted as $1$ when $m=0$. Fix $r\ge3$ and define the polynomial
\begin{equation}\label{eq:Bhat}
 \widehat B_r(x)=\binom{x+r-1}{r}-\sum_{j=0}^{r}(r-j+1)\binom{x-1}{j}+r.
\end{equation}
For every integer $k\ge2$, Theorem~\ref{thm:main} gives $\widehat B_r(k)=\beta(k,k+r)$. We use the notation $B_r(k)=\beta(k,k+r)$ only for the topologically defined sequence $k\ge2$ and retain the hat when discussing the polynomial extension.

\begin{lemma}[Degree and leading coefficient]\label{lem:degree}
For every $r\ge3$, the polynomial $\widehat B_r(x)$ has degree exactly $r-1$, and
\begin{equation}\label{eq:leading-coefficient}
 [x^{r-1}]\widehat B_r(x)=\frac{r-2}{(r-1)!}.
\end{equation}
\end{lemma}

\begin{proof}
Only the terms $\binom{x+r-1}{r}$ and $-\binom{x-1}{r}$ in \eqref{eq:Bhat} have degree $r$. Both have leading coefficient $1/r!$, so their degree-$r$ terms cancel.

For $m\ge1$ and any constant $a$, expansion of the falling factorial gives
\begin{equation}\label{eq:binomial-next-coefficient}
 [x^{m-1}]\binom{x+a}{m}=\frac{a-(m-1)/2}{(m-1)!}.
\end{equation}
Indeed, the coefficient of $x^{m-1}$ in $\prod_{i=0}^{m-1}(x+a-i)$ is $\sum_{i=0}^{m-1}(a-i)=ma-m(m-1)/2$, and division by $m!$ gives \eqref{eq:binomial-next-coefficient}.

Applying \eqref{eq:binomial-next-coefficient} with $(m,a)=(r,r-1)$ gives $[x^{r-1}]\binom{x+r-1}{r}=(r-1)/(2(r-1)!)$. Applying it with $(m,a)=(r,-1)$ and including the minus sign gives $[x^{r-1}](-\binom{x-1}{r})=(r+1)/(2(r-1)!)$. Among the remaining terms in \eqref{eq:Bhat}, only $-2\binom{x-1}{r-1}$ has degree $r-1$, and its leading coefficient is $-2/(r-1)!$. Therefore
\[
 [x^{r-1}]\widehat B_r(x)=\frac{r-1}{2(r-1)!}+\frac{r+1}{2(r-1)!}-\frac{2}{(r-1)!}
 =\frac{r-2}{(r-1)!}.
\]
This coefficient is nonzero for $r\ge3$, so $\deg\widehat B_r=r-1$.
\end{proof}

\subsection{The recurrence and its sharpness}

For a polynomial or a bi-infinite function $f$, let $(\nabla f)(x)=f(x)-f(x-1)$ be the backward difference. Pascal's identity gives $\nabla\binom{x+a}{m}=\binom{x+a-1}{m-1}$ for $m\ge1$. Together with $\nabla\binom{x+a}{0}=0$, iteration yields, for every integer $s\ge0$,
\begin{equation}\label{eq:iterated-difference}
 \nabla^s\binom{x+a}{m}=
 \begin{cases}
 \binom{x+a-s}{m-s}, & 0\le s\le m,\\[1mm]
 0, & s>m.
 \end{cases}
\end{equation}

\begin{theorem}[Exact finite differences]\label{thm:exact-differences}
For every $r\ge3$, the polynomial $\widehat B_r$ satisfies
\begin{equation}\label{eq:polynomial-differences}
 \nabla^r\widehat B_r(x)=0,
 \qquad
 \nabla^{r-1}\widehat B_r(x)=r-2
\end{equation}
identically in $x$. Consequently, for the actual Betti sequence $B_r(k)=\beta(k,k+r)$,
\begin{equation}\label{eq:topological-differences}
 \nabla^r B_r(k)=0\quad(k\ge r+2),
 \qquad
 \nabla^{r-1}B_r(k)=r-2\quad(k\ge r+1).
\end{equation}
\end{theorem}

\begin{proof}
Separate the $j=r$ term in \eqref{eq:Bhat}:
\[
 \widehat B_r(x)=\binom{x+r-1}{r}-\binom{x-1}{r}
 -\sum_{j=0}^{r-1}(r-j+1)\binom{x-1}{j}+r.
\]
By \eqref{eq:iterated-difference}, the $r$th differences of the first two terms are respectively $1$ and $-1$. Every binomial polynomial with lower index at most $r-1$, as well as the constant $r$, is annihilated by $\nabla^r$. Hence $\nabla^r\widehat B_r=1-1=0$.

For the next lower difference, \eqref{eq:iterated-difference} gives $\nabla^{r-1}\binom{x+r-1}{r}=\binom{x}{1}=x$, $\nabla^{r-1}\binom{x-1}{r}=\binom{x-r}{1}=x-r$, and $\nabla^{r-1}\binom{x-1}{r-1}=1$. All terms with lower index at most $r-2$, as well as the constant term, vanish after $r-1$ differences. Since the coefficient of $\binom{x-1}{r-1}$ is $-2$, it follows that $\nabla^{r-1}\widehat B_r(x)=x-(x-r)-2=r-2$. This proves \eqref{eq:polynomial-differences}.

For the topological sequence, $B_r(k-i)$ is defined whenever $k-i\ge2$. Thus the expression $\nabla^rB_r(k)$ is defined for $k\ge r+2$, while $\nabla^{r-1}B_r(k)$ is defined for $k\ge r+1$. On these ranges, every term agrees with the corresponding value of $\widehat B_r$, so \eqref{eq:topological-differences} follows.
\end{proof}

Expanding the first identity in \eqref{eq:topological-differences} gives the strengthened form of the conjecture.

\begin{corollary}[Recurrence conjecture]\label{cor:recurrence}
For every fixed $r\ge3$ and every $k\ge r+2$,
\begin{equation}\label{eq:recurrence}
 \beta(k,k+r)=\sum_{i=1}^{r}(-1)^{i-1}\binom{r}{i}\beta(k-i,k-i+r).
\end{equation}
In particular, \eqref{eq:original-recurrence} holds on the range $k\ge r+3$ stated in \cite[Conjecture 4.10]{BayerCutII2025}.
\end{corollary}

\begin{proof}
For $k\ge r+2$, all indices $k-i$ in \eqref{eq:recurrence} are at least $2$. The identity $\nabla^rB_r(k)=0$ therefore expands as
\[
 B_r(k)-\binom r1B_r(k-1)+\binom r2B_r(k-2)-\cdots+(-1)^rB_r(k-r)=0.
\]
Solving for $B_r(k)$ gives \eqref{eq:recurrence}. At the new boundary $k=r+2$, the last term is $\beta(2,r+2)$. Substituting $k=2$ and $n=r+2$ into \eqref{eq:main-formula} gives $\beta(2,r+2)=(r+1)-[(r+1)+r]+r=0$, so the recurrence remains valid at this endpoint.
\end{proof}

The second identity in \eqref{eq:polynomial-differences} also proves sharpness. Indeed, if $\nabla^s\widehat B_r=0$ for some $0\le s\le r-1$, then applying $\nabla^{r-1-s}$ would give $\nabla^{r-1}\widehat B_r=0$, contradicting $\nabla^{r-1}\widehat B_r=r-2\ne0$. Thus no lower pure power of $\nabla$ annihilates $\widehat B_r$. Nor can such a lower power annihilate the topological sequence eventually: if $\nabla^sB_r(k)=0$ for all sufficiently large $k$, then the polynomial $\nabla^s\widehat B_r(x)$ has infinitely many integer zeros and must vanish identically. Algebraically, the degree-$r$ terms of $\binom{x+r-1}{r}$ and $\binom{x-1}{r}$ cancel, whereas the remaining degree-$(r-1)$ coefficient, including the contribution from $-2\binom{x-1}{r-1}$, is nonzero. The size-$(k+1)$ bad-complement term contributes only the constant $r$ to \eqref{eq:Bhat}; every positive-order difference annihilates this term, so it does not determine the recurrence order.

\subsection{Generating functions}

All generating functions in this subsection are interpreted as formal power series. The finite-difference statement has a precise one-sided generating-function formulation. Let $(a_k)_{k\ge1}$ be any sequence, extend it by $a_k=0$ for $k\le0$, and put $A(x)=\sum_{k\ge1}a_kx^k$. Then, for every $k\ge1$,
\begin{equation}\label{eq:gf-coefficient-difference}
 [x^k](1-x)^rA(x)=\sum_{i=0}^{r}(-1)^i\binom ri a_{k-i}=\nabla^ra_k.
\end{equation}
Consequently, for an integer $K\ge1$, the relation $\nabla^ra_k=0$ for every $k\ge K$ holds if and only if $(1-x)^rA(x)$ is a polynomial of degree at most $K-1$. For a rational generating function $A(x)$, merely requiring its reduced denominator to divide $(1-x)^r$ implies only eventual validity of the recurrence; a bound on the numerator degree is needed to determine the least eventual starting index. The degree condition is essential: $A(x)=x^{r+1}/(1-x)^r$ has denominator $(1-x)^r$, but the coefficient of $x^{r+1}$ in $(1-x)^rA(x)$ is $1$, so $\nabla^r a_{r+1}\ne0$.

Now define $G_r(x)=\sum_{k\ge1}\widehat B_r(k)x^k$. The standard identity $\sum_{k\ge1}\binom{k+r-1}{r}x^k=x/(1-x)^{r+1}$ applies to the first term. For each $j\ge0$, one likewise has $\sum_{k\ge1}\binom{k-1}{j}x^k=x^{j+1}/(1-x)^{j+1}$. Substitution gives
\begin{equation*}
 G_r(x)=\frac{x}{(1-x)^{r+1}}
 -\sum_{j=0}^{r}(r-j+1)\frac{x^{j+1}}{(1-x)^{j+1}}
 +\frac{rx}{1-x}.
\end{equation*}
Multiplying by $(1-x)^r$ and separating the $j=r$ term yields
\begin{align}
 H_r(x):=(1-x)^rG_r(x)
 &=x(1+x+\cdots+x^{r-1})+rx(1-x)^{r-1}\notag\\
 &\quad-\sum_{j=0}^{r-1}(r-j+1)x^{j+1}(1-x)^{r-j-1}.
 \label{eq:H-polynomial}
\end{align}
Thus $H_r(x)$ is a polynomial of degree at most $r$, and $G_r(x)=H_r(x)/(1-x)^r$. Moreover, evaluating \eqref{eq:H-polynomial} at $x=1$ gives $H_r(1)=r-2$: the first term contributes $r$, only the summand $j=r-1$ survives in the final sum, and it contributes $2$. Hence the reduced denominator is exactly $(1-x)^r$ for $r\ge3$, in agreement with $\deg\widehat B_r=r-1$.

The first examples are
\begin{equation*}
 G_3(x)=\frac{x^3}{(1-x)^3},\qquad
 G_4(x)=\frac{x^3(3-x)}{(1-x)^4},\qquad
 G_5(x)=\frac{x^3(6-5x+2x^2)}{(1-x)^5}.
\end{equation*}
It remains to identify the least eventual starting index of this one-sided recurrence. Since $\binom{r-1}{r}=0$ and $\binom{-1}{j}=(-1)^j$, evaluation of \eqref{eq:Bhat} at $0$ gives
\begin{equation*}
 \widehat B_r(0)=r-\sum_{j=0}^{r}(r-j+1)(-1)^j
 =\left\lfloor\frac{r-1}{2}\right\rfloor.
\end{equation*}
Indeed, if $r=2q+1$, pairing adjacent terms gives $q+1$, whereas if $r=2q$, the same pairing gives $q$ together with the final term $1$; in both cases the alternating sum is $\lceil(r+1)/2\rceil$. In particular, $\widehat B_r(0)>0$ for $r\ge3$.

To distinguish the zero extension from the polynomial extension, set $a_k^{(r)}=\widehat B_r(k)$ for $k\ge1$ and $a_k^{(r)}=0$ for $k\le0$. The coefficient of $x^r$ in $H_r(x)$ is then
\begin{align*}
 [x^r]H_r(x)
 &=\sum_{i=0}^{r}(-1)^i\binom ri a_{r-i}^{(r)}
 =\sum_{i=0}^{r-1}(-1)^i\binom ri\widehat B_r(r-i)\\
 &=-(-1)^r\widehat B_r(0)
 =(-1)^{r+1}\left\lfloor\frac{r-1}{2}\right\rfloor.
\end{align*}
The penultimate equality is obtained from the polynomial identity $\nabla^r\widehat B_r(r)=0$ by deleting its $i=r$ term, namely $(-1)^r\widehat B_r(0)$. Hence $[x^r]H_r(x)\ne0$, and therefore $\deg H_r=r$. By \eqref{eq:gf-coefficient-difference}, the zero-extended sequence $(a_k^{(r)})_{k\in\mathbb Z}$ satisfies $\nabla^r a_k^{(r)}=0$ for every $k\ge r+1$, whereas $\nabla^r a_r^{(r)}\ne0$. Consequently, $r+1$ is the least integer $K$ for which the recurrence holds for every $k\ge K$. The polynomial identity in Theorem~\ref{thm:exact-differences} is stronger because it uses the full polynomial extension rather than the zero extension implicit in a one-sided generating function.

\section{Combinatorial and algebraic consequences}\label{sec:consequences}

\subsection{\texorpdfstring{Explicit diagonal and fixed-$k$ formulas}{Explicit diagonal and fixed-k formulas}}

Expanding \eqref{eq:Bhat} gives the first diagonal polynomials
\begin{align*}
 \widehat B_3(k)&=\frac{k^2-3k+2}{2}=\binom{k-1}{2},\\
 \widehat B_4(k)&=\frac{2k^3-3k^2-5k+6}{6},\\
 \widehat B_5(k)&=\frac{3k^4-2k^3+9k^2-58k+48}{24},\\
 \widehat B_6(k)&=\frac{k^5+15k^3-30k^2-46k+60}{30}.
\end{align*}
For example,
\[
 \widehat B_4(k)=\binom{k+3}{4}-5-4(k-1)-3\binom{k-1}{2}
 -2\binom{k-1}{3}-\binom{k-1}{4}+4,
\]
which simplifies to the displayed formula for $\widehat B_4(k)$. Its fourth backward difference therefore vanishes; the same conclusion holds for every fixed $r$ by Theorem~\ref{thm:exact-differences}.

For comparison with the data in \cite{BayerCutII2025}, the first entries are
\begin{center}
\begin{tabular}{c|rrrrrrrr}
\toprule
$r\backslash k$ & $3$ & $4$ & $5$ & $6$ & $7$ & $8$ & $9$ & $10$\\
\midrule
$3$ & $1$ & $3$ & $6$ & $10$ & $15$ & $21$ & $28$ & $36$\\
$4$ & $3$ & $11$ & $26$ & $50$ & $85$ & $133$ & $196$ & $276$\\
$5$ & $6$ & $25$ & $67$ & $145$ & $275$ & $476$ & $770$ & $1182$\\
$6$ & $10$ & $46$ & $136$ & $324$ & $674$ & $1274$ & $2240$ & $3720$\\
\bottomrule
\end{tabular}
\end{center}

We next keep $k$ fixed. For $k=4$ and $n\ge7$, Theorem~\ref{thm:main} gives
\begin{align*}
 \beta(4,n)
 &=\binom{n-1}{3}-\bigl[(n-3)+3(n-4)+3(n-5)+(n-6)\bigr]+(n-4)\\
 &=\binom{n-1}{3}-7n+32.
\end{align*}
Put $u=n-7$. Vandermonde's identity gives $\binom{u+6}{3}=20+15u+6\binom u2+\binom u3$. Since $\binom{n-1}{3}-7n+32=\binom{u+6}{3}-7u-17$, one obtains \eqref{eq:intro-k4}.

For $k=5$ and $n\ge8$, the same calculation gives
\begin{align*}
 \beta(5,n)
 &=\binom{n-1}{4}-\bigl[(n-4)+4(n-5)+6(n-6)+4(n-7)+(n-8)\bigr]+(n-5)\\
 &=\binom{n-1}{4}-15n+91.
\end{align*}
With $v=n-8$, Vandermonde's identity yields $\binom{v+7}{4}=35+35v+21\binom v2+7\binom v3+\binom v4$. Because $\binom{n-1}{4}-15n+91=\binom{v+7}{4}-15v-29$, this proves \eqref{eq:intro-k5}. Thus both identities in \cite[Conjecture 4.11]{BayerCutII2025} hold on their full stated ranges.

\subsection{Layerwise description of nonfaces}

Assume $2\le k\le n-2$ and put $r=n-k\ge2$. For $1\le s\le r$, set $I_s=\{s,s+1,\ldots,s+k\}$. By Proposition~\ref{prop:bad-profile}, the nonfaces of cardinality $r-1$ are exactly
\begin{equation}\label{eq:M-rminus1}
 \mathcal M_{r-1}=\{[n]\setminus I_s:1\le s\le r\},
 \qquad |\mathcal M_{r-1}|=r.
\end{equation}

For $0\le j\le\min\{k-1,r\}$, let $\mathcal Z_j$ be the family of connected $k$-subsets $C=\{c_1<\cdots<c_k\}$ having exactly $j$ gaps $c_{i+1}-c_i$ equal to $2$. Lemma~\ref{lem:connected-count} gives $|\mathcal Z_j|=\binom{k-1}{j}(r-j+1)$. The nonfaces of cardinality $r$ decompose disjointly as
\begin{equation}\label{eq:M-r}
 \mathcal M_r=\bigsqcup_{j=0}^{\min\{k-1,r\}}
 \{[n]\setminus C:C\in\mathcal Z_j\}.
\end{equation}
Every $C\in\mathcal Z_j$ has span $\max C-\min C+1=k+j$, so $j$ records both the number of long gaps and the excess of the span over the cardinality.

Equations \eqref{eq:M-rminus1}--\eqref{eq:M-r} give the layerwise form of the face enumerator:
\[
 F_{k,n}(x)=\sum_{p=0}^{r}\binom np x^p-|\mathcal M_{r-1}|x^{r-1}-|\mathcal M_r|x^r.
\]
When $r\ge3$, evaluation at $x=-1$ and Theorem~\ref{thm:external} give the equivalent correction formula
\begin{equation}\label{eq:beta-layered}
 \beta(k,n)=\binom{n-1}{r}+|\mathcal M_{r-1}|-|\mathcal M_r|.
\end{equation}
Thus the $r$ interval complements contribute the constant $+r$, whereas the connected $k$-complements contribute the subtraction $-z_{k,n}$. In the diagonal polynomial, the constant interval contribution has no effect on any positive-order finite difference.

\subsection{\texorpdfstring{The $h$-polynomial and Hilbert series}{The h-polynomial and Hilbert series}}

Let $\Delta=\Delta_k(P_n^2)$, assume $k\ge2$ and $n\ge k+3$, and put $r=n-k$. Since $\Delta$ has dimension $r-1$, its $h$-polynomial is $h_{\Delta}(t)=(1-t)^rF_{k,n}(t/(1-t))=\sum_{i=0}^{r}h_it^i$. Under this substitution, the correction term $-rx^{r-1}$ becomes $-rt^{r-1}(1-t)$, while $-z_{k,n}x^r$ becomes $-z_{k,n}t^r$. Thus \eqref{eq:face-enumerator} gives
\begin{equation}\label{eq:h-polynomial}
 h_{\Delta}(t)=\sum_{p=0}^{r}\binom np t^p(1-t)^{r-p}
 -rt^{r-1}+(r-z_{k,n})t^r.
\end{equation}
Consequently, for $0\le i\le r$,
\begin{equation}\label{eq:h-coefficients}
 h_i=\sum_{p=0}^{i}(-1)^{i-p}\binom np\binom{r-p}{i-p}
 -r\,\mathbf 1_{\{i=r-1\}}+(r-z_{k,n})\,\mathbf 1_{\{i=r\}}.
\end{equation}
Here $\mathbf 1_{\{E\}}$ is $1$ when the condition $E$ holds and $0$ otherwise. For the top coefficient, \eqref{eq:partial-binomial-sum} gives
\begin{align*}
 h_r
 &=\sum_{p=0}^{r}(-1)^{r-p}\binom np+r-z_{k,n}\\
 &=\binom{n-1}{r}+r-z_{k,n}=\beta(k,n).
\end{align*}
This agrees with the standard Cohen--Macaulay interpretation for pure shellable complexes; see \cite{Baclawski1979,BjornerTopological,WachsPoset}.

Let $S=\kk[x_1,\ldots,x_n]$ and let $I_{\Delta}$ be the Stanley--Reisner ideal. The Hilbert series of $\kk[\Delta]=S/I_{\Delta}$ is therefore
\begin{equation}\label{eq:hilbert-series}
 \Hilb(\kk[\Delta],t)=\frac{1}{(1-t)^r}
 \left(\sum_{p=0}^{r}\binom np t^p(1-t)^{r-p}
 -rt^{r-1}+(r-z_{k,n})t^r\right).
\end{equation}
Equivalently, $\Hilb(\kk[\Delta],t)=F_{k,n}(t/(1-t))$. Thus the complementary-set enumeration determines not only the top Betti number but also the full $f$-vector, $h$-vector, and Stanley--Reisner Hilbert series.

\section{Conclusion}\label{sec:conclusion}

The proof reduces the topology of $\Delta_k(P_n^2)$ to a complete extremal classification of complementary sets. For $2\le k\le n-2$ and every integer $m$, the number of $k$-bad $m$-subsets is zero for $m<k$, equals the number $z_{k,n}$ of connected $k$-subsets when $m=k$, equals $n-k$ when $m=k+1$, and is zero for $m\ge k+2$. This profile gives the full face enumerator \eqref{eq:face-enumerator}, the reduced Euler characteristic \eqref{eq:euler}, and the exact Betti formula \eqref{eq:main-formula}.

For fixed $r=n-k\ge3$, the diagonal values are restrictions of the polynomial $\widehat B_r$ of degree exactly $r-1$. Hence $\nabla^r\widehat B_r=0$ and $\nabla^{r-1}\widehat B_r=r-2$. On the actual Betti sequence, the $r$th-order recurrence is valid for every $k\ge r+2$, which is precisely the full range on which all shifted Betti numbers are defined without extending the parameter $k$ below $2$. The one-sided generating function has reduced denominator $(1-x)^r$ and a numerator of degree exactly $r$; equivalently, $r+1$ is the least eventual starting index for the zero-extended recurrence. The same enumeration proves the conjectured formulas for $k=4,5$ and determines the associated Stanley--Reisner enumerators.

\section*{Declaration of Generative AI and AI-Assisted Technologies in the Writing Process}
During the preparation of this article, the authors used DeepSeek to construct a specialized agent for mathematical problem solving and to produce an initial proof of the main theorem. The authors subsequently checked and revised the output and take full responsibility for the content of the article.

\end{document}